\documentclass{article}
\usepackage[british]{babel}
\usepackage{microtype}
\usepackage{fontenc}[T1]
\usepackage{amssymb}
\usepackage{array}
\usepackage{latexsym}
\usepackage{enumerate}
\usepackage{amsmath}
\usepackage{amsfonts}
\usepackage{amsthm}
\usepackage{geometry}
\usepackage{microtype}

\newcommand{\cC}{\mathcal C}
\newcommand{\cL}{\mathcal L}
\newcommand{\cP}{\mathcal P}
\newcommand{\cH}{\mathcal H}
\newcommand{\cU}{\mathcal U}

\newcommand{\GF}{\operatorname{GF}}
\newcommand{\AG}{\operatorname{AG}}
\newcommand{\PG}{\operatorname{PG}}
\newcommand{\PGL}{\operatorname{PGL}}
\newcommand{\PGU}{\operatorname{PGU}}
\newcommand{\PSU}{\operatorname{PSU}}

\theoremstyle{plain}
\newtheorem{prop}{Proposition}[section]
\newtheorem{theorem}[prop]{Theorem}

\newtheorem{lemma}[prop]{Lemma}
\theoremstyle{definition}

\title{Unitals in $\PG(2,q^2)$ with a large 2-point stabiliser}
\author{L.Giuzzi and G.Korchm\'aros}
\begin{document}
\maketitle
\begin{abstract}
Let $\cU$ be a unital embedded in the Desarguesian projective plane
$\PG(2,q^2)$. Write $M$ for the subgroup of $\PGL(3,q^2)$ which preserves
$\cU$. We show that $\cU$ is classical if and only if $\cU$ has two
distinct points $P,Q$ for which the stabiliser $G=M_{P,Q}$ has order
$q^2-1$.
\end{abstract}
\section{Introduction}
In the Desarguesian projective plane $\PG(2,q^2)$, a \emph{unital} is
defined to be a set of $q^3+1$ points containing either $1$ or $q+1$
points from each line of $\PG(2,q^2)$.  Observe that each unital has a
unique $1$-secant at each of its points.
The idea of a unital arises from the combinatorial properties of
the non-degenerate unitary
polarity $\pi$ of $\PG(2,q^2)$. The set of absolute points of $\pi$ is
indeed a unital, called the {\em classical} or {\em Hermitian}
unital. Therefore, the projective group preserving the classical
unital is isomorphic to $\PGU(3,q)$ and acts on
its points as $\PGU(3,q)$ in its natural $2$-transitive permutation
representation. Using the classification of subgroups of
$\PGL(3,q^2)$, Hoffer \cite{hof} proved that a unital is classical if and
only if if is preserved by a collineation group isomorphic to
$\PSU(3,q^2)$. Hoffer's characterisation has been the starting point
for several investigations of unitals in terms of the structure of
their automorphism group,
see \cite{Ba,Bi,BK1,BK2,CEK1,CEK2,EW,Do,LG2,Ka1,Ka2}; see
also the survey \cite[Appendix B]{BE}.
In $\PG(2,q^2)$ with $q$ odd, L.M.~Abatangelo \cite{A1} proved
that a Buekenhout--Metz unital with a cyclic $2$--point stabiliser of
order $q^2-1$ is necessarily classical.
In their talk at Combinatorics 2010, G.~Donati e N.~Durante have
conjectured that Abatangelo's characterisation holds true for any
unital in $\PG(2,q^2)$. In this note, we provide a proof of this
conjecture.

Our notation and terminology are standard, see \cite{BE}, and \cite{hkt}.
We shall assume $q>2$, since all unitals in $\PG(2,4)$
are classical.

\section{Some technical lemmas}
 Let $M$ be the subgroup of $\PGL(3,q^2)$ which preserves a unital
 $\cU$ in $\PG(2,q^2)$. A {\em{$2$-point stabiliser of}} $\cU$ is a
 subgroup of $M$ which fixes two distinct points of $\cU$.
 \begin{lemma}
   \label{l1} Let $\cU$ be a unital in $\PG(2,q^2)$
   with a $2$--point
   stabiliser $G$ of order $q^2-1$. Then, $G$ is cyclic, and there
   exists a projective frame in $\PG(2,q^2)$ such that $G$ is generated
   by a projectivity with matrix representation
   \[ \begin{pmatrix}
     \lambda & 0 & 0 \\
     0 & \mu & 0 \\
     0 & 0 & 1
   \end{pmatrix}, \]
   where $\lambda$ is a primitive element of
   $\GF(q^2)$ and $\mu$ is a primitive element of $\GF(q)$.
\end{lemma}
\begin{proof}
  Let $O,Y_\infty$ be two distinct points of $\cU$ such
  that the stabiliser $G=M_{O,Y_\infty}$ has order $q^2-1$.  Choose a
  projective frame in $\PG(2,q^2)$ so that $O=(0,0,1)$,
  $Y_{\infty}=(0,1,0)$ and the $1$-secants of $\cU$ at those points are
  respectively $\ell_X: X_2=0$ and $\ell_{\infty}: X_3=0$.
  Write
  $X_{\infty}=(1,0,0)$ for the common point of $\ell_X$ and
  $\ell_{\infty}$. Observe that $G$ fixes the vertices of the triangle
  $OX_\infty Y_\infty$.  Therefore, $G$ consists of
  projectivities with diagonal matrix representation.
  Let now $h\in G$ be a
  projectivity that fixes  a further point
  $P\in\ell_X$ apart from $O,X_\infty$.
  Then, $h$ fixes $\ell_X$ point-wise; that is, $h$ is a
  perspectivity with axis $\ell_X$. Since $h$ also fixes  $Y_\infty$, the
  centre of $h$ must  be $Y_\infty$.
  Take any point $R\in \ell_X$ with $R\neq O,X_\infty$.
  Obviously, $h$ preserves the line $r=Y_\infty R$;
  hence, it also preserves $r\cap \cU$.
  Since $r\cap\cU$ comprises $q$
  points other than $R$, the subgroup $H$ generated by $h$ has a
  permutation representation of degree $q$ in which no non-trivial
  permutation fixes a point. As $q=p^r$ for a prime $p$, this implies
  that $p$ divides $|H|$. On the other hand, $h$ is taken
  from a group of order  $q^2-1$.
  Thus, $h$ must be the trivial element in $G$. Therefore,
  $G$ has a faithful action on $\ell_X$ as a $2$-point stabiliser of
  $\PG(1,q^2)$. This proves that $G$ is cyclic. Furthermore, a generator
  $g$ of $G$ has a matrix representation
  \[
  \begin{pmatrix}
    \lambda & 0 & 0 \\
    0 & \mu & 0 \\
    0 & 0 & 1
  \end{pmatrix}\
  \text{\rm with $\lambda$  a primitive element of $\GF(q^2)$}.
  \]
  As $G$ preserves the set $\Delta=\cU\cap OY_{\infty}$, it also
  induces a permutation group $\bar{G}$ on $\Delta$. Since any
  projectivity fixing three points of $OY_{\infty}$ must fix
  $OY_{\infty}$ point-wise, $\bar{G}$ is semiregular on
  $\Delta$.
  Therefore, $|\bar{G}|$ divides $q-1$.
  Let now $F$ be the
  subgroup of $G$ fixing $\Delta$ point-wise. Then, $F$ is a
  perspectivity group with centre $X_\infty$ and axis $\ell_Y:X_1=0$.
  Take any point $R\in \ell_Y$ such that the line $r=RX_\infty$ is
  a $(q+1)$-secant of $\cU$. Then, $r\cap \cU$ is disjoint from $\ell_Y$.
  Hence,
  $F$ has a permutation representation on $r\cap \cU$ in
  which no non-trivial permutation fixes a point. Thus, $|F|$
  divides $q+1$.
  Since $|G|=q^2-1$, we have $|\bar{G}|\leq q-1$ and $|G|=|\bar{G}||F|$.
  This implies
  $|\bar{G}|=q-1$ and $|F|=q+1$. From the former condition, $\mu$ must be
  a primitive element of $\GF(q)$.
\end{proof}
\begin{lemma}
\label{lem2} In $\PG(2,q^2)$, let $\cH_1$ and $\cH_2$ be two
non-degenerate Hermitian curves which have the same tangent at a
common point $P$. Denote by $I(P,\cH_1\cap\cH_2)$ the intersection
multiplicity of $\cH_1$ and $\cH_2$ at $P$ Then,
\begin{equation}
\label{eq0} I(P,\cH_1\cap \cH_2)=q+1.
\end{equation}
\end{lemma}
\begin{proof} Since, up to projectivities, there is a unique class
  of
  Hermitian curves in $\PG(2,q^2)$, we may assume $\cH_1$ to have
  equation $-X_1^{q+1}+X_2^qX_3^{}+X_2^{}X_3^q=0$. Furthermore, as the
  projectivity group $\PGU(3,q)$ preserving $\cH_1$ acts transitively on
  the points of $\cH_1$ in $\PG(2,q^2)$, we may also suppose
  $P=(0,0,1)$. Within this setting, the tangent $r$ of $\cH_1$ at $P$
  coincides with the line $X_2=0$. As no term $X_1^j$ with $0<j\leq q$
  occurs in the equation of $\cH_1$, the intersection multiplicity
  $I(P,\cH_1\cap r)$ is equal to $q+1$.

The equation of the other Hermitian curve $\cH_2$ might be written as
\[ F(X_1,X_2,X_3)=
a_0\,X_3^qX_2^{}+a_1X_3^{q-1}G_1(X_1,X_2)+\ldots+a_qG_q(X_1,X_2)=0,\]
where $a_0\neq 0$ and $\deg\,G_i(X_1,X_2)=i+1$. Since the tangent of
$\cH_2$ at $P$ has no other common point with $\cH_2$, even over the
algebraic closure of $\GF(q^2)$, no terms $X_1^j$ with $0<j\leq q$ can
occur in the polynomials $G_i(X_1,X_2)$. In other words,
$I(P,\cH_2\cap r)=q+1$.

 A primitive representation of the unique branch of $\cH_1$ centred at
$P$ has components $$x(t)=t,\,\,y(t)=ct^i+\ldots,\,\,x_3(t)=1$$ where
$i$ is a positive integer and $y(t)\in \GF(q^2)[[t]]$, that is,
$y(t)$ stands for a formal power series with coefficients in
$\GF(q^2)$.

 {}From $I(P,\cH_1\cap r)=q+1$,
 $$y(t)^q+y(t)-t^{q+1}=0,$$ whence
 $y(t)=t^{q+1}+H(t)$,
 where $H(t)$ is a formal power series of order at least $q+2$.
 That is, the exponent $j$ in the leading term $ct^j$ of $H(t)$
 is larger than $q+1$.

 It is now possible to compute
 the intersection multiplicity $I(P,\cH_1\cap \cH_2)$
 using \cite[Theorem 4.36]{hkt}:
 \[ I(P,\cH_1\cap \cH_2)=
 {\rm{ord}}_t\,F(t,y(t),1)={\rm{ord}}_t\,(a_0t^{q+1}+G(t)), \]
with $G(t)\in \GF(q^2)[[t]]$ of order at least $q+2$. From this, the
assertion follows.
\end{proof}
\begin{lemma}
\label{lem3} In $\PG(2,q^2)$, let $\cH$ be a non-degenerate Hermitian
curve and let $\cC$ be a Hermitian cone whose centre does not lie on
$\cH$.  Assume that there exist two points $P_i\in \cH\cap \cC$, with
$i=1,2$, such that the tangent line of $\cH$ at $P_i$ is a linear
component of $\cC$. Then
\begin{equation}
\label{eq3} I(P_1,\cH\cap \cC)=q+1.
\end{equation}
\end{lemma}
\begin{proof} We use the same setting as in the proof of Lemma
\ref{lem2} with $P=P_1$. Since the action of $\PGU(3,q)$ is
2-transitive on the points of $\cH$, we may also suppose that
$P_2=(0,1,0)$. Then the centre of $\cC$ is the point
$X_\infty=(1,0,0)$, and $\cC$ has equation
$c^qX_2^qX_3^{}+cX_2^{}X_3^q=0$ with $c\neq 0$. Therefore,
 $$I(P,\cH\cap \cC)={\rm{ord}}_t\,(c^qy(t)^q+cy(t))={\rm{ord}}_t\,(c^qt^{q+1}+K(t))$$
 with $K(t)\in \GF(q^2)[[t]]$ of order at least $q+2$, whence the
assertion follows.
\end{proof}
\section{Main result}
\begin{theorem}
\label{main} In $\PG(2,q^2),$ let $\cU$ be a unital and write $M$ for the
group of projectivities which preserves $\cU$.  If $\cU$ has two
distinct points $P,Q$ such that the stabiliser\, $G=M_{P,Q}$ has order
$q^2-1$, then $\cU$ is classical.
\end{theorem}
The main idea of the proof is to build up a projective plane
of order $q$ using, for the definition of  points,
non-trivial $G$-orbits in the affine plane $\AG(2,q^2)$
which arise from $\PG(2,q^2)$ by removing the line $\ell_{\infty}:X_3=0$
with all its points. To
this purpose, take $\cU$ and $G$ as in Lemma \ref{l1},
with $\mu=\lambda^{q+1}$,
\\
\noindent
and define an
incidence structure $\Pi=(\cP,\cL)$ as follows:
\begin{enumerate}
\item
  Points are all non-trivial $G$-orbits in $\AG(2,q^2)$.
\item
  Lines are $\ell_Y$, and the
  non-degenerate Hermitian curves of equation
\begin{equation}
\label{eq1} \cH_b:\,\, -X_1^{q+1}+bX_3^{}X_2^q+b^qX_3^qX_2^{}=0,
\end{equation} with $b$ ranging over $\GF(q^2)^*$, together with the
Hermitian cones of equation
\begin{equation}
\label{eq2} \cC_c:\,\, c^qX_2^qX_3^{}+cX_2^{}X_3^q=0,
\end{equation}
with $c$ ranging over a representative system of cosets
of $(\GF(q),*)$ in $(\GF(q^2),*)$.
\item Incidence is the natural inclusion.
\end{enumerate}
\begin{lemma}
\label{lem1} The incidence structure
$\Pi=(\cP,\cL)$ is a projective plane of order $q$.
\end{lemma}
\begin{proof} In $\AG(2,q^2)$, the group $G$ has $q^2+q+1$ non-trivial
orbits, namely its $q^2$ orbits disjoint from $\ell_Y$, each of
length $q^2-1$, and its $q+1$
orbits on $\ell_Y$, these of length $q-1$. Therefore, the total number
of points in $\cP$ is equal to $q^2+q+1$. By construction of
$\Pi$, the number of lines in $\cL$ is also  $q^2+q+1$.
Incidence is well defined as $G$ preserves $\ell_Y$ and each
Hermitian curve and cone representing lines of $\cL$.

We now count the
points incident with a line in $\Pi$. Each $G$-orbit on $\ell_Y$
distinct from $O$ and $Y_\infty$ has length $q-1$. Hence there are
exactly $q+1$ such $G$-orbits; in terms of $\Pi$, the line represented
by $\ell_Y$ is incident with $q+1$ points.
A Hermitian curve $\cH_b$ of Equation \eqref{eq1} has $q^3$ points
in $\AG(2,q^2)$  and
meets $\ell_Y$ in a $G$-orbit, while it contains no point from the line
$\ell_X$. As $q^3-q=q(q^2-1)$, the line represented by $\cH_b$ is
incident with $q+1$ points in $\cP$.
Finally, a Hermitian
cone $\cC_c$ of Equation \eqref{eq2} has $q^3$ points in $\AG(2,q^2)$
and contains $q$
points from $\ell_Y$. One of these $q$ points is $O$, the other $q-1$
forming a non-trivial $G$-orbit. The remaining $q^3-q$ points of $\cC_c$ are
partitioned into $q$ distinct $G$-orbits. Hence, the line represented by
$\cC_c$ is also incident with $q+1$ points. This shows that  each line in $\Pi$
is incident with exactly $q+1$ points.

Therefore, it is enough to show that two any two distinct lines of
$\cL$ have exactly one common point. Obviously, this is true when one
of these lines is represented by $\ell_Y$. Furthermore, the point of
$\cP$ represented by $\ell_X$ is incident with each line of $\cL$
represented by a Hermitian cone of equation \eqref{eq2}. We are led to
investigate the case where one of the lines of $\cL$ is represented by
a Hermitian curve $\cH_b$ of equation \eqref{eq2}, and the other line
of $\cL$ is represented by a Hermitian curve $\cH$ which is either
another Hermitian curve $\cH_d$ of the same type of Equation \eqref{eq1}, or a
Hermitian cone $\cC_c$ of Equation~\eqref{eq2}.

Clearly, both $O$ and $Y_\infty$ are common points of $\cH_b$ and
$\cH$. From Kestenband's classification \cite{Kest}, see also
\cite[Theorem 6.7]{BE}, $\cH_b\cap \cH$ cannot consist of exactly two
points. Therefore, there exists another point, say $P\in \cH_b\cap
\cH$. Since $\ell_X$ and $\ell_0$ are $1$-secants of $\cH_b$ at the
points $O$ and $Y_\infty$, respectively, either $P$ is on $\ell_Y$ or
$P$ lies outside the fundamental triangle. In the latter case, the
$G$-orbit $\Delta_1$ of $P$ has size $q^2-1$ and represents a point in
$\cP$. Assume that $\cH_b\cap \cH$ contains a further point, not lying
in $\Delta_1$. If the $G$-orbit of $Q$ is $\Delta_2$,
then
\[|\cH_b\cap \cH|\geq |\Delta_1|+|\Delta_2|=2(q^2-1)+2=2q^2.\]
However, from B\'ezout's theorem, see \cite[Theorem 3.14]{hkt},
\[ |\cH_b\cap \cH|\leq (q+1)^2. \]
Therefore, $Q\in \ell_Y$, and the $G$-orbit $\Delta_3$ of $Q$ has
length $q-1$. Hence, $\cH_b$ and $\cH$ shear $q+1$ points on
$\ell_Y$. If $\cH=\cH_d$ is a Hermitian curve of Equation \eqref{eq1}, each
of these $q+1$ points is the tangency point of a common inflection
tangent with multiplicity $q+1$ of the Hermitian curves $\cH_b$ and
$\cH$.
 Write $R_1,\ldots R_{q+1}$ for these points. Then, by \eqref{eq0}
the intersection multiplicity is $I(R_i,\cH_b\cap \cH_d)=q+1$.
 This holds
true also when $\cH$ is a Hermitian cone $\cC_c$ of Equation \eqref{eq2};
see Lemma \ref{lem3}. Therefore, in any case,
\[\sum_{i=1}^{q+1}I(R_i,\cH_b\cap \cH)=(q+1)^2.\]
From B\'ezout's theorem,
$\cH_b\cap \cH=\{R_1,\ldots R_{q+1}\}.$ Therefore,
$\cH_b\cap\cH=\Delta_3\cup \{O,Y_\infty\}$. This shows that if
$Q\not\in \ell_Y$, the lines represented by $\cH_b$ and $\cH$ have
exactly one point in common. The above argument can also be adapted to
prove this assertion in the case where $Q\in \ell_Y$. Therefore, any
two distinct lines of $\cL$ have exactly one common point.
\end{proof}

\begin{proof}[Proof of Theorem \ref{main}]
Assume first $\mu=\lambda^{q+1}$.
Construct a projective plane $\Pi$ as in Lemma \ref{lem1}. Since
$\cU\setminus\{O,Y_\infty \}$ is the union of $G$-orbits, $\cU$
represents a set $\Gamma$ of $q+1$ points in $\Pi$. From \cite{BBW},
$N\equiv 1 \pmod{p}$ where $N$ is the number of common points of $\cU$
with any Hermitian curve $\cH_b$. In terms of $\Pi$, $\Gamma$ contains
some point from every line $\Lambda$ in $\cL$ represented by a
Hermitian curve of Equation \eqref{eq1}. Actually, this holds true
when the line $\Lambda$ in $\cL$ is represented by a Hermitian cone
$\cC$ of Equation \eqref{eq2}. To prove it, observe that $\cC$
contains a line $r$ distinct from both lines $\ell_X$ and
$\ell_0$. Then $r\cap \cU$ is non empty, and contains neither $O$ nor
$Y_\infty$. If $P$ is point in $r\cap \cU$, then the $G$-orbit of $P$
represents a common point of $\Gamma$ and $\Lambda$. Since the line in
$\cL$ represented by $\ell_Y$ meets $\Gamma$, it turns out that
$\Gamma$ contains some point from every line in $\cL$.

  Therefore, $\Gamma$ is itself a line in $\cL$. Note that $\cU$
contains no line. In terms of $\PG(2,q^2)$, this yields that $\cU$
coincides with a Hermitian curve of Equation \eqref{eq1}. In
particular, $\cU$ is a classical unital.


To investigate the case $\mu\neq \lambda^{q+1}$, we stil work in
the above plane $\Pi$.  By a straightforward computation, the
projectivity $g$ given in Lemma \ref{l1} induces a non-trivial
collineation on $\Pi$. Also, $g$ preserves every Hermitian cone of
Equation \eqref{eq2} and the common line $\ell_X$ of these Hermitian
cones. In terms of $\Pi$, $\bar{g}$ is a perspectivity with centre
at the point represented by $\ell_X$. Since $g$ also preserves the
line $\ell_Y$, the axis of $\bar{g}$ is $\ell_Y$, regarded as a line in
$\Pi$. Therefore, every point of $\Pi$ lying on $\ell_Y$ is fixed by
$g$. Consequently, $\bar{g}^{q-1}$ is the identity collineation. As $g$
has order $q^2-1$, this yields that $g^{q+1}$ preserves every
Hermitian curve of Equation \eqref{eq1}. Thus,
$\mu^{q+1}=(\lambda^{q+1})^{q+1}$, whence $\mu=-\lambda^{q+1}$. In
particular, $p\neq 2$.

Consider now the $q+1$ non-trivial $G$-orbits in $\cU$ with $G=\langle g
\rangle$. For any point $P\in \Pi$, let $n_P$ the number of the
non-trivial $G$-orbits in $\cU$ intersecting the set $\rho(P)$
representing $P$ in $PG(2,q^2)$. Then $n_P=1$ when $\rho(P)$ is the
unique $G$-orbit in $\cU$ which lies on $\ell_Y$. Otherwise, $0\leq
n_P \leq 2$, with $n_P=2$ if and only if $\rho(P)$ is not a $G$-orbit
but the union of two $H$-orbits with $H=\langle g^2\rangle$.

Let $\Gamma$ be the multiset in $\Pi$ consisting of all points with
$n_P>0$ and define the weight $\nu_P$ of $P$ to be either $1$ or $2$,
according as $n_P=2$ or $n_P=1$. Then, $\sum_{P\in\Gamma}{\nu_P}=2q+2$.
We show that
$\Gamma$ is a $2$-fold blocking multiset of $\Pi$. For this purpose,
let $\cH$ be either a Hermitian curve of Equation \eqref{eq1} or a
Hermitian cone of Equation \eqref{eq2}. Write $m$ for the number of
common points of $\cH_b$ and $\cU$, different from $O$ and $Y_\infty$;
thus, the total number of common points is $N=m+2$.  As $N\equiv 1 \pmod
p$, we have $m\geq 1$. Take $P\in \cH\cap \cU$. If $\nu_P=2$, then the
line representing $\cH$ meets $\Gamma$ in a point with weight $2$. If
$\nu_P=1$, then the $H$-orbit of $P$ has size $(q^2-1)/2$ and lies on
both $\cH$ and $\cU$.  Since $(q^2-1)/2+2\not\equiv 1 \pmod p$, $\cH$
and $\cU$ must share a further point $Q$ other than $O$ and
$Y_\infty$. Therefore, the points $P'$ and $Q'$ of $\Pi$ which
represent the subsets containing $P$ and $Q$ are distinct. This shows
that $\Gamma$ meets the line represented by $\cH$ in two distinct
points. Therefore, $\Gamma$ is a $2$-fold blocking multiset.

Since $\Gamma$ has at least one point with weight $2$, this yields
that $\Gamma$ comprises of all points of a line, each with weight
$2$. Hence, $\cU$ coincides with the Hermitian curve representing that
line. This is to say that $\cU$ is a classical unital.
\end{proof}

\end{document}